\documentclass[11pt]{amsart}
\usepackage{graphicx}
\usepackage{fullpage}
\usepackage{tikz, xparse, float, xcolor}
\usetikzlibrary{arrows.meta}

\newtheorem{theorem}{Theorem}[section]

\newtheorem{lemma}[theorem]{Lemma}

\newtheorem*{BijThm}{Theorem~\ref{thm:bijections}}

\theoremstyle{definition}
\newtheorem{definition}[theorem]{Definition}

\newcommand{\boundcond}{
\draw[-{Stealth[scale=1.2]}] (2,0.5)--++(-0.35,0); \draw (2,0.5)--++(-0.5,0);
\draw[-{Stealth[scale=1.2]}] (2,1)--++(-0.35,0); \draw (2,1)--++(-0.5,0);
\draw[-{Stealth[scale=1.2]}] (2,1.5)--++(-0.35,0); \draw (2,1.5)--++(-0.5,0);
\draw[-{Stealth[scale=1.2]}] (0,1.5)--++(0.35,0); \draw (0,1.5)--++(0.5,0);
\draw[-{Stealth[scale=1.2]}] (0,1)--++(0.35,0); \draw (0,1)--++(0.5,0);
\draw[-{Stealth[scale=1.2]}] (0,0.5)--++(0.35,0); \draw (0,0.5)--++(0.5,0);
}
\newcommand{\iceI}[2]{
\draw[-{Stealth[scale=1.2]}] (#1,#2)--++(0.35,0); \draw (#1,#2)--++(0.5,0);
\draw[-{Stealth[scale=1.2]}] (#1,#2)--++(0,0.35); \draw (#1,#2)--++(0,0.5);
}
\newcommand{\iceII}[2]{
\draw[-{Stealth[scale=1.2]}] (#1,#2)--++(-0.35,0); \draw (#1,#2)--++(-0.5,0);
\draw[-{Stealth[scale=1.2]}] (#1,#2)--++(0,0.35); \draw (#1,#2)--++(0,0.5);
}
\newcommand{\iceIII}[2]{
\draw[-{Stealth[scale=1.2]}] (#1,#2)--++(-0.35,0); \draw (#1,#2)--++(-0.5,0);
\draw[-{Stealth[scale=1.2]}] (#1,#2)--++(0,-0.35); \draw (#1,#2)--++(0,-0.5);
}
\newcommand{\iceIV}[2]{
\draw[-{Stealth[scale=1.2]}] (#1,#2)--++(0.35,0); \draw (#1,#2)--++(0.5,0);
\draw[-{Stealth[scale=1.2]}] (#1,#2)--++(0,-0.35); \draw (#1,#2)--++(0,-0.5);
}
\newcommand{\iceup}[2]{
\draw[-{Stealth[scale=1.2]}] (#1,#2)--++(0,0.35); \draw (#1,#2)--++(0,0.5);
\draw[-{Stealth[scale=1.2]}] (#1,#2)--++(0,-0.35); \draw (#1,#2)--++(0,-0.5);
}
\newcommand{\iceside}[2]{
\draw[-{Stealth[scale=1.2]}] (#1,#2)--++(0.35,0); \draw (#1,#2)--++(0.5,0);
\draw[-{Stealth[scale=1.2]}] (#1,#2)--++(-0.35,0); \draw (#1,#2)--++(-0.5,0);
}
\newcommand{\magogthree}[2]{
\draw[-{Stealth[scale=1.2]}] (#1,#2)--++(0,0.35); \draw (#1,#2)--++(0,0.5);
\draw[-{Stealth[scale=1.2]}] (#1,#2)--++(0,-0.35); \draw (#1,#2)--++(0,-0.5);
\draw[-{Stealth[scale=1.2]}] (#1,#2)--++(-0.35,0); \draw (#1,#2)--++(-0.5,0);
}
\newcommand{\magogone}[2]{
\draw[-{Stealth[scale=1.2]}] (#1,#2)--++(-0.35,0); \draw (#1,#2)--++(-0.5,0);
}

\title{The Many Faces of Magog Matrices}
\author{Rohan Bansal}
\address{Liberty High School, Frisco, TX.} 
\email{rohanbansal07@gmail.com}

\author{Jessica Striker$^*$}
\address{Department of Mathematics, North Dakota State University, Fargo, ND.} 
\email{jessica.striker@ndsu.edu}
\thanks{$^*$Striker's research is partially supported by NSF Grant DMS-2247089 and Simons Foundation Gift MP-TSM-00002802.}

\begin{document}
\begin{abstract}
Magog matrices, introduced by Holmlund and Striker in 2025, provide a matrix model for totally symmetric self-complementary plane partitions (TSSCPPs), as a natural analogue of alternating sign matrices (ASMs). In this paper, we develop several new combinatorial representations of magog matrices, mirroring classical representations of ASMs. Specifically, we define magog analogues of corner-sum matrices, height-function matrices, fully packed loop configurations, and vertex models, and establish explicit bijections among all of these objects. These constructions provide new structural insight into the combinatorics of TSSCPPs and illuminate  parallels and differences between the ASM and TSSCPP frameworks.
\end{abstract}

\maketitle

\section{Introduction}\label{sec:intro}

Alternating sign matrices (ASMs) are $n \times n$ matrices with entries in $\{0, 1, -1\}$ 
in which the nonzero entries in each row and column alternate in sign, beginning and ending 
with $+1$, and in which each row and column sums to $1$. First introduced by Mills, Robbins, 
and Rumsey \cite{MRR83} in the context of Dodgson condensation, ASMs quickly revealed 
themselves to be rich combinatorial objects admitting a striking variety of equivalent 
representations. The number of $n \times n$ ASMs is given by the product formula
\[
\prod_{j=0}^{n-1} \frac{(3j+1)!}{(n+j)!},
\]
conjectured in \cite{MRR83} and first proved by Zeilberger \cite{Z96a}, then proved 
by Kuperberg \cite{K96} using the six-vertex model from statistical mechanics.

A long-standing open problem in algebraic combinatorics is to find an explicit bijection between 
$n \times n$ ASMs and totally symmetric self-complementary plane partitions (TSSCPPs) in a 
$2n \times 2n \times 2n$ box. These two sets of objects are known to be equinumerous. Andrews 
proved the TSSCPP enumeration formula \cite{A94} in 1991, and Zeilberger's proof of the ASM 
conjecture the following year \cite{Z96a} established the equinumerosity, yet no satisfying explicit 
bijection between them is known. Several partial bijections have been constructed 
\cite{CB11, Bet15, HS24}, and the problem remains a driving motivation in the study 
of ASMs and plane partitions.

In \cite{HS25}, Holmlund and Striker introduced a new class of matrices, called 
\emph{magog matrices}, which provide a matrix representation of TSSCPPs. An $n \times n$ magog matrix is a square $\{0, 1, -1\}$-matrix satisfying a subset of the ASM row 
and column sum conditions, together with an additional set of inequalities on 
partial sums called the \emph{special inequalities}; see Definition~\ref{def:magog} for specifics. The set of $n \times n$ magog matrices 
is in explicit bijection with TSSCPPs in a $2n \times 2n \times 2n$ box. Thus, magog 
matrices sit naturally alongside ASMs as a TSSCPP-side representation of the 
equinumerosity mystery.

The title of this paper is inspired by Propp's survey \cite{P01}, which catalogues many 
equivalent combinatorial representations of ASMs, including corner-sum matrices, 
height-function matrices, three-colorings, monotone triangles, order ideals of a certain 
poset, square ice states (six-vertex models), gasket-and-basket tilings, and fully packed 
loop (FPL) configurations. Each of these representations illuminates a different facet of 
the combinatorics of ASMs.  Zeilberger defined \emph{magog triangles}~\cite{Z96a} as a TSSCPP analogue of monotone triangles. TSSCPPs have also been shown to have an order ideal interpretation~\cite{tetra,RS} and a representation as a certain type of Boolean triangle~\cite{S18}. Our goal in this paper is to develop  representations of TSSCPPs analogous to other ASM objects discussed in \cite{P01}.

The main result of this paper is the following, proved in Section~\ref{sec:m_bij}.

\begin{BijThm}
There are explicit bijections among magog matrices, magog corner-sum 
matrices, magog height-function matrices, magog fully-packed loop 
configurations, and magog vertex models all of order $n$.
\end{BijThm}

This paper is organized as follows. In Section~\ref{sec:prelim}, we review the necessary 
background on magog matrices and TSSCPPs and recall relevant results 
from \cite{HS25}. In Section~\ref{subsec:ASM_faces}, we review the suite of ASM representations from \cite{P01} that 
will serve as templates for our constructions. In Section~\ref{sec:m_bij}, we define each 
of the new magog analogues and prove Theorem~\ref{thm:bijections}.

\section{TSSCPPs and magog matrices}\label{sec:prelim}

In this section we review the definitions of TSSCPPs and magog matrices. We also mention relevant enumerative results on magog matrices from \cite{HS25}.

\begin{definition}[\cite{Mac01,S86}]
A \emph{plane partition} $\pi$ is a set of lattice points with positive integer coordinates
$(i,j,k)$ such that if $(i,j,k) \in \pi$ and $1 \leq i' \leq i$, $1 \leq j' \leq j$,
$1 \leq k' \leq k$, then $(i',j',k') \in \pi$. Equivalently, it is a weakly decreasing
two-dimensional array of nonnegative integers. A plane partition fits \emph{inside a box}
$a \times b \times c$ if all its nonzero entries satisfy $i \leq a$, $j \leq b$,
$k \leq c$.
\end{definition}

\begin{definition}[\cite{S86}]
A plane partition $\pi$ is \emph{totally symmetric} if whenever $(i,j,k) \in \pi$, all
six permutations of $(i,j,k)$ are also in $\pi$. It is \emph{self-complementary} inside
an $a \times b \times c$ box if the set of lattice points $(i,j,k)$ with $1 \leq i
\leq a$, $1 \leq j \leq b$, $1 \leq k \leq c$ and $(i,j,k) \notin \pi$ is a translated copy
of $\pi$ (that is, the ``empty space'' has the same shape as the ``filled space''). A
\emph{totally symmetric self-complementary plane partition (TSSCPP)} inside a $2n \times
2n \times 2n$ box is a plane partition that is both totally symmetric and
self-complementary inside the box.
\end{definition}

We now define square sign matrices, which are a superset of both magog matrices and alternating sign matrices.

\begin{definition}[{\cite[Definition~2.2]{HS25}}]\label{def:ssm}
An $n \times n$ \emph{square sign matrix} is a matrix $A = (a_{i,j})$ with entries in
$\{0, 1, -1\}$ satisfying:
\begin{itemize}
    \item $\displaystyle\sum_{j=1}^n a_{i,j} = 1$ for all $1 \leq i \leq n$ (rows sum
    to $1$),
    \item $\displaystyle\sum_{i=1}^n a_{i,j} = 1$ for all $1 \leq j \leq n$ (columns
    sum to $1$),
    \item $\displaystyle\sum_{i'=1}^{i} a_{i',j} \in \{0,1\}$ for all $i,j$ (partial
    column sums are $0$ or $1$), and 
    \item $\displaystyle\sum_{j'=1}^{j} a_{i,j'} \geq 0$ for all $i,j$ (partial row sums
    are nonnegative).
\end{itemize}
We denote the set of $n \times n$ square sign matrices by $\mathrm{Sign}_n$.
\end{definition}

Alternating sign matrices are the square sign matrices additionally satisfying
$\sum_{j'=1}^{j} a_{i,j'} \leq 1$; that is, the partial row sums lie
in $\{0,1\}$. Magog matrices are a different subclass, characterized by a distinct additional
inequality.

\begin{definition}[{\cite[Definition~3.1]{HS25}}]\label{def:magog}
An \emph{magog matrix} of order $n$ is a square sign matrix $A = (a_{i,j})$ such that
\[
\sum_{j'=1}^{j} a_{i+1,j'} + \sum_{i'=1}^{i+1} a_{i',j+1} - \sum_{i'=1}^{i} a_{i',j}
\geq 0
\]
for all $1 \leq i \leq n-2$, $1 \leq j \leq n-2$. We call this the
\emph{$(i,j)$-special inequality} and write $\mathrm{Magog}_n$ for the set of all $n
\times n$ magog matrices.
\end{definition}

The left-hand side of the special inequality consists of the partial row
sum of row $i+1$ through column $j$, plus the partial column sum of column $j+1$ through
row $i+1$, minus the partial column sum of column $j$ through row $i$. The special inequality is violated
precisely when the partial column sum of column $j$ up to row $i$ equals $1$, yet both the
partial row sum of row $i+1$ up to column $j$ and the partial column sum of column $j+1$
through row $i+1$ are zero.

Note that $\mathrm{ASM}_n$ and $\mathrm{Magog}_n$ are distinct subsets of
$\mathrm{Sign}_n$. For instance, every permutation matrix lies in $\mathrm{ASM}_n$, while only the
$132$-avoiding permutation matrices lie in $\mathrm{Magog}_n$ \cite[Theorem~3.7]{HS25}.
The last matrix in Figure~1 of \cite{HS25} shows that neither set contains the other.

Holmlund and Striker established a number of refined enumerations of $\mathrm{Magog}_n$. The maximum
number of $-1$ entries in any $n \times n$ magog matrix is $\lfloor \frac{n-1}{2}
\rfloor\lceil \frac{n-1}{2}\rceil$, which is the same as for ASMs. Refined enumerations by the
position of the $1$ in the first or last row or column, and by inversion number were also given in \cite{HS25}.

\section{Alternating sign matrix bijections}\label{subsec:ASM_faces}

We now recall the suite of equivalent representations of ASMs surveyed by Propp
\cite{P01} and extended to the partial ASM setting by Heuer \cite{H24}. These will serve
as the templates for our magog analogues in Section~\ref{sec:m_bij}. We follow the
graph-theoretic setup of \cite{H24} to give uniform definitions that will adapt cleanly
to the magog setting.

\begin{definition}
\label{def:csm}
Given an ASM $(a_{i,j})_{1 \leq i,j \leq n}$, the associated \emph{corner-sum matrix}
is the $(n+1) \times (n+1)$ matrix $(c_{i,j})_{0 \leq i,j \leq n}$ defined by
$c_{i,j} = \sum_{i' \leq i,\, j' \leq j} a_{i',j'}$. Intrinsically, corner-sum matrices
are characterized as matrices $(c_{i,j})_{0 \leq i,j \leq n}$ satisfying:
\begin{itemize}
    \item $c_{0,k} = c_{k,0} = 0$ and $c_{n,k} = c_{k,n} = k$ for all $0 \leq k \leq n$,
    \item $c_{i,j} - c_{i,j-1} \in \{0,1\}$ (rows increase by at most $1$), and
    \item $c_{i,j} \geq c_{i-1,j}$ (columns are weakly increasing).
\end{itemize}
\end{definition}

\begin{definition}
\label{def:hfm}
The \emph{height-function matrix} of an ASM is the $(n+1) \times (n+1)$ matrix
$(h_{i,j})_{0 \leq i,j \leq n}$ defined from the corner-sum matrix by $h_{i,j} =
i + j - 2c_{i,j}$. Intrinsically, height-function matrices are characterized by:
\begin{itemize}
    \item the first row and column are $0, 1, 2, \ldots, n$ (consecutively),
    \item the last row and column are $n, n-1, \ldots, 0$ (consecutively), and
    \item any two row-adjacent or column-adjacent entries differ by exactly $1$.
\end{itemize}
\end{definition}

The height-function matrix can be viewed as a proper $3$-coloring of the $(n+1) \times
(n+1)$ grid by reducing entries modulo $3$ \cite{P01}.

We now define the graph underlying the fully-packed loop and six-vertex model representations,
following the setup of \cite{H24}.

\begin{definition}
\label{def:Gmn}
Define the graph $G_{m,n}$ as follows. The vertex set is
\[
V_{m,n} = \{v_{i,j} : 0 \leq i \leq m+1,\, 0 \leq j \leq n+1\}
\setminus \{v_{0,0},\, v_{0,n+1},\, v_{m+1,0},\, v_{m+1,n+1}\}.
\]
The \emph{interior vertices} are $\{v_{i,j} : 1 \leq i \leq m,\, 1 \leq j \leq n\}$;
the remaining vertices are \emph{boundary vertices}. A vertex $v_{i,j}$ is \emph{even}
(resp.\ \emph{odd}) if $i+j$ is even (resp.\ odd). The edge set is
\[
E_{m,n} = \{v_{i,j}v_{i+1,j} : 0 \leq i \leq m,\, 1 \leq j \leq n\}
\cup \{v_{i,j}v_{i,j+1} : 1 \leq i \leq m,\, 0 \leq j \leq n\}.
\]
\end{definition}

In the square case $m = n$, the graph $G_{n,n}$ is the underlying graph for both the FPL
and six-vertex model representations of ASMs.

\begin{definition}
\label{def:FPL}
A \emph{fully-packed loop (FPL) configuration} of order $n$ is a subgraph of $G_{n,n}$
in which every interior vertex has exactly two incident edges, and the boundary edges are
selected according to the following \emph{alternating boundary conditions}: when $n$ is
odd, edges $v_{0,j}v_{1,j}$ (for $j$ odd) and $v_{n,j}v_{n+1,j}$ (for $j$ odd) and
$v_{i,0}v_{i,1}$ (for $i$ even) and $v_{i,n}v_{i,n+1}$ (for $i$ even) are included;
when $n$ is even, edges $v_{0,j}v_{1,j}$ (for $j$ odd), $v_{n,j}v_{n+1,j}$ (for $j$
even), $v_{i,0}v_{i,1}$ (for $i$ even), and $v_{i,n}v_{i,n+1}$ (for $i$ odd) are
included.
\end{definition}

\begin{definition}
\label{def:sqice}
A \emph{square ice configuration} (or \emph{six-vertex model state}) of order $n$ is an
orientation of the edges of $G_{n,n}$ such that each interior vertex has exactly two
incoming and two outgoing edges. We impose \emph{domain-wall boundary conditions}: all
arrows along the left and right boundaries point inward, and all arrows along the top and
bottom boundaries point outward. 
\end{definition}

FPL configurations are in bijection with square ice configurations (hence with ASMs)
via the rule that retains exactly those edges oriented from an even vertex to an odd
vertex, discards the rest, then forget the orientations \cite[Lemma~23]{H24}. The boundary conditions
match up correctly under this correspondence.

The seven FPL and square ice configurations for $n = 3$ are shown in Figures 9 and 14
of \cite{P01}. In Section~\ref{sec:m_bij}, we define magog analogues of corner-sum
matrices, height-function matrices, fully packed loops, and vertex models, and prove
explicit bijections among all five representations of magog matrices.

\section{Magog Matrix Bijections}\label{sec:m_bij}

In the previous section, we defined \textit{magog matrices}. In this section, we construct four new combinatorial objects using analogues of the bijections found in the  alternating sign matrix setting. Our main theorem is the following. 

\begin{theorem}\label{thm:bijections}
There are explicit bijections among the following: 
\begin{itemize}
    \item $n\times n$ magog matrices, 
    \item magog corner-sum matrices of order $n$, 
    \item magog height-function matrices of order $n$, 
    \item magog fully-packed loop configurations of order $n$, and
    \item magog vertex models of order $n$.
\end{itemize}
\end{theorem}

We prove this theorem using a series of definitions followed by lemmas. We will start by proving a bijection between magog matrices and magog corner-sum matrices.

\begin{definition}
\label{def:magog_corner}
An \emph{magog corner-sum matrix of order $n$} is a matrix $(c_{i,j})_{0 \leq i,j \leq n}$ with the following properties:
\begin{itemize}
    \item $c_{0,k}=c_{k,0}=0$ for $0 \leq k \leq n$,
    \item $c_{n,k}=c_{k,n}=k$ for $0 \leq k \leq n$,
    \item $c_{i,j}-c_{i,j-1} \in \{0,1\}$,
    \item $c_{i,j} \geq c_{i-1,j}$, and
    \item \textbf{Special inequality: }$c_{i+1,j+1}+c_{i,j-1} \geq 2c_{i,j}$.
\end{itemize}
\end{definition}

\begin{figure} [htbp]
    \centering
\[
\begin{pmatrix}
    0 & 0 & 0 & 0 \\
    0 & 0 & 0 & 1 \\
    0 & 0 & 1 & 2 \\
    0 & 1 & 2 & 3
\end{pmatrix}
\ \ \begin{pmatrix}
    0 & 0 & 0 & 0 \\
    0 & 0 & 0 & 1 \\
    0 & 1 & 1 & 2 \\
    0 & 1 & 2 & 3
\end{pmatrix}
\ \ \begin{pmatrix}
    0 & 0 & 0 & 0 \\
    0 & 0 & 1 & 1 \\
    0 & 0 & 1 & 2 \\
    0 & 1 & 2 & 3
\end{pmatrix}
\ \ \begin{pmatrix}
    0 & 0 & 0 & 0 \\
    0 & 0 & 1 & 1 \\
    0 & 1 & 1 & 2 \\
    0 & 1 & 2 & 3
\end{pmatrix}
\]
\[
\begin{pmatrix}
    0 & 0 & 0 & 0 \\
    0 & 0 & 0 & 1 \\
    0 & 1 & 2 & 2 \\
    0 & 1 & 2 & 3
\end{pmatrix}
\ \ \begin{pmatrix}
    0 & 0 & 0 & 0 \\
    0 & 0 & 1 & 1 \\
    0 & 1 & 2 & 2 \\
    0 & 1 & 2 & 3
\end{pmatrix}
\ \ \begin{pmatrix}
    0 & 0 & 0 & 0 \\
    0 & 1 & 1 & 1 \\
    0 & 1 & 2 & 2 \\
    0 & 1 & 2 & 3
\end{pmatrix}
\]
    \caption{The seven magog corner-sum matrices of order 3}
    \label{fig:enter-label}
\end{figure}

\begin{lemma}
\label{lem:cs}
There is an explicit bijection between $n \times n$ magog matrices and magog corner-sum matrices of order $n$.
\end{lemma}

\begin{proof}
    Let $(a_{i,j})$ be a magog matrix. We start by defining the matrix $(c_{i,j})_{0\leq i \leq n,0\leq j \leq n}$ by setting $c_{i,j} = \displaystyle\sum_{i'=1}^i \displaystyle\sum_{j'=1}^j a_{i',j'}$. The matrix $(c_{i,j})$ satisfies Definition~\ref{def:magog_corner} because of the following. The entries of the first row and first column are zero by definition. Since the sum of every row in a magog matrix is one, we know $c_{i,n}=i$. Similarly, $c_{n,j}=j$ because the sum of every column in a magog matrix is one. Since the nonzero entries in every column of a magog matrix alternate in sign, and thus, the column-partial sum of any cell is either zero or one, we know the rows of $(c_{i,j})$ are weakly increasing by at most one. By the same logic, since the left partial sum of any cell in a magog matrix is non-negative, the columns of our corner-sum matrix must be weakly increasing. Lastly, we derive the special inequality for corner-sum matrices by starting with the special inequality for magog matrices and adding terms to both sides to create  double summations.
    \begin{align*}
        \sum_{j'=1}^j a_{i+1,j'} + \sum_{i'=1}^{i+1} a_{i',j+1} - \sum_{i'=1}^i a_{i',j} &\geq 0 \\
        \sum_{j'=1}^j a_{i+1,j'} + \sum_{i'=1}^{i+1} a_{i',j+1} &\geq \sum_{i'=1}^i a_{i',j} \\
        \sum_{j'=1}^j a_{i+1,j'} + \sum_{i'=1}^{i+1} a_{i',j+1} +
        \sum_{i'=1}^i \sum_{j'=1}^j a_{i',j'} &\geq \sum_{i'=1}^i a_{i',j} +\sum_{i'=1}^i \sum_{j'=1}^j a_{i',j'}\\
        \sum_{i'=1}^{i+1} \sum_{j'=1}^{j+1} a_{i',j'} &\geq \sum_{i'=1}^i \sum_{j'=1}^j a_{i,j} - \sum_{i'=1}^i \sum_{j'=1}^{j-1} a_{i',j'} + \sum_{i'=1}^i \sum_{j'=1}^j a_{i',j'}
    \end{align*}
    Now, we apply the definition of $c_{i,j}$:
    \begin{align*}
        c_{i+1,j+1} &\geq c_{i,j} - c_{i,j-1} + c_{i,j} \\
        c_{i+1,j+1} + c_{i,j-1} &\geq 2c_{i,j}
    \end{align*}
    which gives the desired result.
\end{proof}

\begin{definition}
\label{def:magog_height}
An \emph{magog height-function matrix of order $n$} is a matrix $(h_{i,j})_{0 \leq i,j \leq n}$ with the following properties:
\begin{itemize}
    \item $h_{0,k}=h_{k,0}=k$ for $0 \leq k \leq n$,
    \item $h_{n,k}=h_{k,n}=n-k$ for $0 \leq k \leq n$,
    \item $h_{i,j}-h_{i,j-1} =\pm1$,
    \item $h_{i,j} - h_{i-1,j} \in \{1-2k, \text{for } k \geq 0\}$, and
    \item \textbf{Special inequality: }$h_{i+1,j+1}+h_{i,j-1} \leq 2h_{i,j}+1$.
\end{itemize}
\end{definition}

\begin{figure} [H]
    \centering
$$
\begin{pmatrix}
    0 & 1 & 2 & 3 \\
    1 & 2 & 3 & 2 \\
    2 & 3 & 2 & 1 \\
    3 & 2 & 1 & 0
\end{pmatrix}
\ \ \begin{pmatrix}
    0 & 1 & 2 & 3 \\
    1 & 2 & 3 & 2 \\
    2 & 1 & 2 & 1 \\
    3 & 2 & 1 & 0
\end{pmatrix}
\ \ \begin{pmatrix}
    0 & 1 & 2 & 3 \\
    1 & 2 & 1 & 2 \\
    2 & 3 & 2 & 1 \\
    3 & 2 & 1 & 0
\end{pmatrix}
\ \ \begin{pmatrix}
    0 & 1 & 2 & 3 \\
    1 & 2 & 1 & 2 \\
    2 & 1 & 2 & 1 \\
    3 & 2 & 1 & 0
\end{pmatrix}
$$
$$
\begin{pmatrix}
    0 & 1 & 2 & 3 \\
    1 & 2 & 3 & 2 \\
    2 & 1 & 0 & 1 \\
    3 & 2 & 1 & 0
\end{pmatrix}
\ \ \begin{pmatrix}
    0 & 1 & 2 & 3 \\
    1 & 2 & 1 & 2 \\
    2 & 1 & 0 & 1 \\
    3 & 2 & 1 & 0
\end{pmatrix}
\ \ \begin{pmatrix}
    0 & 1 & 2 & 3 \\
    1 & 0 & 1 & 2 \\
    2 & 1 & 0 & 1 \\
    3 & 2 & 1 & 0
\end{pmatrix}
$$
    \caption{The seven magog height-function matrices of order 3}
    \label{fig:enter-label}
\end{figure}

\begin{lemma}
\label{lem:hf}
There is an explicit bijection between $n\times n$ magog matrices and magog height-function matrices of order $n$.
\end{lemma}

\begin{proof}
We begin by mapping an $n\times n$ magog matrix $(a_{i,j})_{0< i\leq n,0<j\leq n }$ to its corresponding magog corner-sum matrix $(c_{i,j})$. Then, we define $(h_{i,j})_{0< i\leq n,0<j\leq n }$ as $h_{i,j}=i+j-2c_{i,j}$. The matrix $(h_{i,j})$ satisfies Definition~\ref{def:magog_height} because of the following. The condition that the first row and the first column of a magog corner-sum matrix are $0$ guarantees that the first row and first column start at $0$ and increase by one. Since the last row and last column of a magog corner-sum matrix start from 0 and increase by 1, we know $h_{n,j}=n+j-2c_{n,j}=n+j-2j=n-j$ and $h_{i,n}=i+n-2c_{i,n}=i+n-2i=n-i$. The condition that the entries of the corner-sum matrix weakly increase by at most one from left to right is equivalent to the condition that adjacent entries in ($h_{i,j}$) differ by exactly one. Similarly, the condition that the entries of the corner-sum matrix weakly increase from top to bottom is equivalent to the condition that the adjacent entries of a magog height-function matrix increase by 1 or decrease by an odd number from top to bottom. Lastly, we derive the special inequality for magog height-function matrices by starting with the special inequality for magog corner-sum matrices. We start with the special inequality of magog corner-sum matrices, double both sides, then add terms to both sides. This gives us the special inequality of magog height-function matrices.
    \begin{align*}
        c_{i+1,j+1}+c_{i,j-1} &\geq 2 c_{i,j} \\
        -2c_{i+1,j+1}-2c_{i,j-1} &\leq -4 c_{i,j} \\
        2i+2j+1-2c_{i+1,j+1}-2c_{i,j-1}&\leq 2i+2j+1-4 c_{i,j} \\
        i+1+j+1-2c_{i+1,j+1}+i+j-1-2c_{i,j-1}&\leq 2(i+j-2 c_{i,j})+1 \\
        h_{i+1,j+1}+h_{i,j-1}&\leq 2h_{i,j}+1 
    \end{align*}
\end{proof}

For the next definition, recall the graph $G_{m,n}$ from Definition~\ref{def:Gmn}.
\begin{definition}
     A \emph{magog fully packed loop configuration of order} $n$ is a subgraph of $G_{n,n}$ with the following properties:
     \begin{itemize}
         \item the boundary edges follow alternating boundary conditions
         \item the internal boundaries have degree 1, 2, or 3
         \item the only allowed odd degree structures are shown below where the parity of a vertex equals the parity of $i+j$
         \begin{center}
         \begin{tikzpicture} [scale=0.4]

\draw[thick] (0,0) rectangle (8,-8);
\draw[thick] (4,0) -- (4,-8);
\draw[thick] (0,-4) -- (8,-4);

\node at (2,1) {\tiny Right-Facing};
\node at (6,1) {\tiny Left-Facing};
\node[rotate=90] at (-1,-2) {Even};
\node[rotate=90] at (-1,-6) {Odd};

\fill (2, -2) circle (2pt);
\fill (1, -2) circle (2pt);
\fill (3, -2) circle (2pt);
\fill (2, -1) circle (2pt);
\fill (2, -3) circle (2pt);
\draw[thick] (1, -2) -- (2, -2);

\fill (6, -2) circle (2pt);
\fill (5, -2) circle (2pt);
\fill (7, -2) circle (2pt);
\fill (6, -1) circle (2pt);
\fill (6, -3) circle (2pt);
\draw[thick] (6, -1) -- (6, -3);
\draw[thick] (5, -2) -- (6, -2);

\fill (2, -6) circle (2pt);
\fill (1, -6) circle (2pt);
\fill (3, -6) circle (2pt);
\fill (2, -5) circle (2pt);
\fill (2, -7) circle (2pt);
\draw[thick] (2, -6) -- (3, -6);
\draw[thick] (2, -5) -- (2, -7);

\fill (6, -6) circle (2pt);
\fill (5, -6) circle (2pt);
\fill (7, -6) circle (2pt);
\fill (6, -5) circle (2pt);
\fill (6, -7) circle (2pt);
\draw[thick] (6, -6) -- (7, -6);

\end{tikzpicture}
        \end{center}
         \item for every $i,j$, [$\#$ of left-facing vertices from $(i,0)$ to $(i,j)$] minus [$\#$ of right-facing vertices from $(i,0)$ to $(i,j)$] $\geq 0$
         \item \textbf{Special forbidden structure:} Magog FPLs cannot contain either of the follow structures.
         \begin{center}
         \begin{tikzpicture}[scale=0.5]
\fill (-1, 0) circle (2pt);
\fill (-2, 0) circle (2pt);
\fill (-3, 0) circle (2pt);
\fill (-2, 1) circle (2pt);
\fill (-2, -1) circle (2pt);
\fill (-3, 1) circle (2pt);
\draw[fill=none] (-3, 1) circle (0.33) node [yshift=0.5cm] {Odd};
\draw[thick] (-2, 1) -- (-2, 0);
\draw[thick] (-2, 0) -- (-1, 0);
\draw[thick] (-3, 1) -- (-3, 0);

\fill (3, 0) circle (2pt);
\fill (2, 0) circle (2pt);
\fill (1, 0) circle (2pt);
\fill (2, 1) circle (2pt);
\fill (1, 1) circle (2pt);
\fill (2, -1) circle (2pt);
\draw[fill=none] (1, 1) circle (0.33) node [yshift=0.5cm] {Even};
\draw[thick] (1, 0) -- (2, 0);
\draw[thick] (2, -1) -- (2, 0);
    
\end{tikzpicture}
    \end{center}
     \end{itemize}
\end{definition}

\begin{figure} [H]
    \centering
    \begin{tikzpicture}[scale=0.5]
\foreach \a in {0,...,2}{
    \fill[black] (\a,3) circle (2pt);
}
\foreach \a in {0,...,2}{
    \fill[black] (\a,-1) circle (2pt);
}
\foreach \a in {0,...,2}{
    \fill[black] (3,\a) circle (2pt);
}
\foreach \a in {0,...,2}{
    \fill[black] (-1,\a) circle (2pt);
}

\foreach \x in {0,...,2}{
  \foreach \y in {0,...,2}{
    \fill[black] (\x,\y) circle (2pt);
  }
}
\draw[thick] (0,0)--++(0,-1);
\draw[thick] (2,0)--++(0,-1);
\draw[thick] (0,2)--++(0,1);
\draw[thick] (2,2)--++(0,1);
\draw[thick] (0,1)--++(-1,0);
\draw[thick] (2,1)--++(1,0);
\draw[thick] (0,2)--++(1,0)--++(0,-2)--++(1,0);
\draw[thick] (0,0)--++(0,1);
\draw[thick] (2,1)--++(0,1);
\end{tikzpicture}
\ \ \
\begin{tikzpicture}[scale=0.5]
\foreach \a in {0,...,2}{
    \fill[black] (\a,3) circle (2pt);
}
\foreach \a in {0,...,2}{
    \fill[black] (\a,-1) circle (2pt);
}
\foreach \a in {0,...,2}{
    \fill[black] (3,\a) circle (2pt);
}
\foreach \a in {0,...,2}{
    \fill[black] (-1,\a) circle (2pt);
}

\foreach \x in {0,...,2}{
  \foreach \y in {0,...,2}{
    \fill[black] (\x,\y) circle (2pt);
  }
}
\draw[thick] (0,0)--++(0,-1);
\draw[thick] (2,0)--++(0,-1);
\draw[thick] (0,2)--++(0,1);
\draw[thick] (2,2)--++(0,1);
\draw[thick] (0,1)--++(-1,0);
\draw[thick] (2,1)--++(1,0);
\draw[thick] (0,1)--++(1,0)--++(0,1)--++(-1,0);
\draw[thick] (0,0)--++(2,0);
\draw[thick] (2,1)--++(0,1);
\end{tikzpicture}
\ \ \
\begin{tikzpicture}[scale=0.5]
\foreach \a in {0,...,2}{
    \fill[black] (\a,3) circle (2pt);
}
\foreach \a in {0,...,2}{
    \fill[black] (\a,-1) circle (2pt);
}
\foreach \a in {0,...,2}{
    \fill[black] (3,\a) circle (2pt);
}
\foreach \a in {0,...,2}{
    \fill[black] (-1,\a) circle (2pt);
}

\foreach \x in {0,...,2}{
  \foreach \y in {0,...,2}{
    \fill[black] (\x,\y) circle (2pt);
  }
}
\draw[thick] (0,0)--++(0,-1);
\draw[thick] (2,0)--++(0,-1);
\draw[thick] (0,2)--++(0,1);
\draw[thick] (2,2)--++(0,1);
\draw[thick] (0,1)--++(-1,0);
\draw[thick] (2,1)--++(1,0);
\draw[thick] (0,0)--++(0,1);
\draw[thick] (2,0)--++(-1,0)--++(0,1)--++(1,0);
\draw[thick] (0,2)--++(2,0);
\end{tikzpicture}
\ \ \
\begin{tikzpicture}[scale=0.5]
\foreach \a in {0,...,2}{
    \fill[black] (\a,3) circle (2pt);
}
\foreach \a in {0,...,2}{
    \fill[black] (\a,-1) circle (2pt);
}
\foreach \a in {0,...,2}{
    \fill[black] (3,\a) circle (2pt);
}
\foreach \a in {0,...,2}{
    \fill[black] (-1,\a) circle (2pt);
}

\foreach \x in {0,...,2}{
  \foreach \y in {0,...,2}{
    \fill[black] (\x,\y) circle (2pt);
  }
}
\draw[thick] (0,0)--++(0,-1);
\draw[thick] (2,0)--++(0,-1);
\draw[thick] (0,2)--++(0,1);
\draw[thick] (2,2)--++(0,1);
\draw[thick] (0,1)--++(-1,0);
\draw[thick] (2,1)--++(1,0);
\draw[thick] (0,0)--++(2,0);
\draw[thick] (0,1)--++(2,0);
\draw[thick] (0,2)--++(2,0);
\end{tikzpicture}
$$
$$
\begin{tikzpicture}[scale=0.5]
\foreach \a in {0,...,2}{
    \fill[black] (\a,3) circle (2pt);
}
\foreach \a in {0,...,2}{
    \fill[black] (\a,-1) circle (2pt);
}
\foreach \a in {0,...,2}{
    \fill[black] (3,\a) circle (2pt);
}
\foreach \a in {0,...,2}{
    \fill[black] (-1,\a) circle (2pt);
}

\foreach \x in {0,...,2}{
  \foreach \y in {0,...,2}{
    \fill[black] (\x,\y) circle (2pt);
  }
}
\draw[thick] (0,0)--++(0,-1);
\draw[thick] (2,0)--++(0,-1);
\draw[thick] (0,2)--++(0,1);
\draw[thick] (2,2)--++(0,1);
\draw[thick] (0,1)--++(-1,0);
\draw[thick] (2,1)--++(1,0);
\draw[thick] (0,0)--++(1,0)--++(0,2)--++(-1,0);
\draw[thick] (2,0)--++(0,2);
\draw[thick] (0,1)--++(1,0);
\end{tikzpicture}
\ \ \
\begin{tikzpicture}[scale=0.5]
\foreach \a in {0,...,2}{
    \fill[black] (\a,3) circle (2pt);
}
\foreach \a in {0,...,2}{
    \fill[black] (\a,-1) circle (2pt);
}
\foreach \a in {0,...,2}{
    \fill[black] (3,\a) circle (2pt);
}
\foreach \a in {0,...,2}{
    \fill[black] (-1,\a) circle (2pt);
}

\foreach \x in {0,...,2}{
  \foreach \y in {0,...,2}{
    \fill[black] (\x,\y) circle (2pt);
  }
}
\draw[thick] (0,0)--++(0,-1);
\draw[thick] (2,0)--++(0,-1);
\draw[thick] (0,2)--++(0,1);
\draw[thick] (2,2)--++(0,1);
\draw[thick] (0,1)--++(-1,0);
\draw[thick] (2,1)--++(1,0);
\draw[thick] (0,2)--++(2,0);
\draw[thick] (0,0)--++(1,0)--++(0,1)--++(-1,0);
\draw[thick] (2,0)--++(0,1);
\end{tikzpicture}
\ \ \
\begin{tikzpicture}[scale=0.5]
\foreach \a in {0,...,2}{
    \fill[black] (\a,3) circle (2pt);
}
\foreach \a in {0,...,2}{
    \fill[black] (\a,-1) circle (2pt);
}
\foreach \a in {0,...,2}{
    \fill[black] (3,\a) circle (2pt);
}
\foreach \a in {0,...,2}{
    \fill[black] (-1,\a) circle (2pt);
}

\foreach \x in {0,...,2}{
  \foreach \y in {0,...,2}{
    \fill[black] (\x,\y) circle (2pt);
  }
}
\draw[thick] (0,0)--++(0,-1);
\draw[thick] (2,0)--++(0,-1);
\draw[thick] (0,2)--++(0,1);
\draw[thick] (2,2)--++(0,1);
\draw[thick] (0,1)--++(-1,0);
\draw[thick] (2,1)--++(1,0);
\draw[thick] (0,0)--++(1,0)--++(0,2)--++(1,0);
\draw[thick] (0,1)--++(0,1);
\draw[thick] (2,0)--++(0,1);
\end{tikzpicture}

    \caption{The seven magog fully packed loops of order 3}
    \label{fig:placeholder}
\end{figure}

\begin{lemma}
\label{lem:fpl}
There is an explicit bijection between magog height function matrices of order $n$ and magog fully packed loops of order $n$.
\end{lemma}

\begin{proof}
     Starting with a magog height function matrix $(h_{i,j})$ of order $n$, overlay $G_{n,n}$ so that each interior number in the partial
height-function matrix has four surrounding vertices. Separate two horizontally adjacent
numbers by an edge if the numbers are $2k$ and $2k+ 1$ (in either order) for any integer $k$.
Separate two vertically adjacent numbers by an edge if the numbers are $2k-1$ and $2k$ (in either order) for any integer $k$. Since $h_{0,k}= h_{k,0}=k$ for $0\leq k\leq n$, and $h_{n,k}= h_{k,n}=n-k$ for $0\leq k\leq n$,
we will have the required alternating boundary conditions for a magog fully packed loop of order $n$.
The adjacency conditions of magog height functions dictate the allowed degrees/structures in the magog fully packed loops. When $[h_{i,j},h_{i,j+1},h_{i+1,j+1},h_{i+1,j}]=[k+3,k+2,k+1,k]$, we get a right-facing structure. When $[h_{i,j},h_{i,j+1},h_{i+1,j+1},h_{i+1,j}]=[k+2,k+3,k,k+1]$, we get a left-facing structure. Every other valid set produces a degree 2 structure. The fourth condition follows from the fact that the row-partial sum of magog matrices must always be non-negative. The special forbidden structure is derived by applying the above map to configurations which violate the special inequality for magog height function matrices. Specifically, the special inequality for magog height-function matrices is  violated when $h_{i+1,j+1}=k+2,h_{i,j-1}=k+1$, and $h_{i,j}=k$.  Applying the bijection to these vertices, we obtain the forbidden structures of the magog fully packed loops. For the reverse map, we start
with the magog fully-packed loop configuration and fill in the boundary conditions of a magog height-function matrix. Then, we use the rule described above in reverse,
using the third and fourth conditions of magog height function matrices. \end{proof}

\begin{definition}
\label{def:magog_vertex}
    A \emph{magog vertex model of order n} is a directed graph whose underlying graph is $G_{m,n}$ with the following properties:
    \begin{itemize}
        \item the graph satisfies domain-wall boundary conditions
        \item each interior vertex has either 1, 2, or 3 incoming arrows
        \item the only allowed odd degree structures are shown below
        
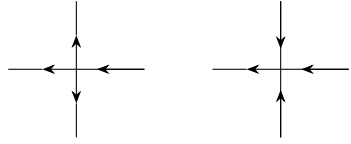
\begin{center}
        \begin{tikzpicture}[scale=0.9]
  \begin{scope}[xshift=0cm]
    \draw[-{Stealth}] (1,0) -- (0.3,0); \draw (0,0)--(1,0);
    \draw[-{Stealth}] (0,0) -- (0,0.5); \draw (0,0.5)--(0,1);
    \draw[-{Stealth}] (0,0) -- (-0.5,0); \draw (-0.5,0)--(-1,0);
    \draw[-{Stealth}] (0,0) -- (0,-0.5); \draw (0,-0.5)--(0,-1);
  \end{scope}
  \begin{scope}[xshift=3cm]
    \draw[-{Stealth}] (0,0) -- (-0.5,0); \draw (-0.5,0)--(-1,0);
    \draw[-{Stealth}] (0,1) -- (0,0.3); \draw (0,0)--(0,1);
    \draw[-{Stealth}] (1,0) -- (0.3,0); \draw (0,0)--(1,0);
    \draw[-{Stealth}] (0,-1) -- (0,-0.3); \draw (0,0)--(0,-1);
  \end{scope}
  \end{tikzpicture}
    \end{center}
        \item for every $i,j$, [$\#$ of degree 1 vertices from $(i,0)$ to $(i,j)$] minus [$\#$ of degree 3 vertices from $(i,0)$ to $(i,j)$] $\geq 0$ 
        \item \textbf{Special forbidden structure:} Magog vertex models cannot contain the following structure.
        \begin{center}
            \begin{tikzpicture}
                \begin{scope}[xshift=0cm]
    \draw[-{Stealth}] (0,0) -- (0.5,0); \draw (0,0)--(1,0);
    \draw[-{Stealth}] (0,0) -- (0,0.5); \draw (0,0.5)--(0,1);
    \draw[-{Stealth}] (-1,0) -- (-0.3,0); \draw (0,0)--(-1,0);
    \draw[-{Stealth}] (0,-1) -- (0,-0.3); \draw (0,0)--(0,-1);
    \draw[-{Stealth}] (-1,1) -- (-1,0.3); \draw (-1,1)--(-1,0);
  \end{scope}
            \end{tikzpicture}
        \end{center}
    \end{itemize}
\end{definition}

\begin{figure} [htbp]
    \centering
    \begin{tikzpicture}[scale=1.2]
        \boundcond;
        \iceup{0.5}{0.5};
        \iceI{0.5}{1};
        \iceI{0.5}{1.5};
        \iceIII{1}{0.5};
        \iceup{1}{1};
        \iceI{1}{1.5};
        \iceIII{1.5}{0.5};
        \iceIII{1.5}{1};
        \iceup{1.5}{1.5};
    \end{tikzpicture}
    \ \ \ 
    \begin{tikzpicture}[scale=1.2]
        \boundcond;
        \iceIV{0.5}{0.5};
        \iceup{0.5}{1};
        \iceI{0.5}{1.5};
        \iceup{1}{0.5};
        \iceII{1}{1};
        \iceI{1}{1.5};
        \iceIII{1.5}{0.5};
        \iceIII{1.5}{1};
        \iceup{1.5}{1.5};
    \end{tikzpicture}
    \ \ \ 
    \begin{tikzpicture}[scale=1.2]
        \boundcond;
        \iceup{0.5}{0.5};
        \iceI{0.5}{1};
        \iceI{0.5}{1.5};
        \iceIII{1}{0.5};
        \iceIV{1}{1};
        \iceup{1}{1.5};
        \iceIII{1.5}{0.5};
        \iceup{1.5}{1};
        \iceII{1.5}{1.5};
    \end{tikzpicture}
    \ \ \ 
    \begin{tikzpicture}[scale=1.2]
        \boundcond;
        \iceIV{0.5}{0.5};
        \iceup{0.5}{1};
        \iceI{0.5}{1.5};
        \iceup{1}{0.5};
        \iceside{1}{1};
        \iceup{1}{1.5};
        \iceIII{1.5}{0.5};
        \iceup{1.5}{1};
        \iceII{1.5}{1.5};
    \end{tikzpicture}
    \\
    \begin{tikzpicture}[scale=1.2]
        \boundcond;
        \iceIV{0.5}{0.5};
        \iceup{0.5}{1};
        \iceI{0.5}{1.5};
        \iceIV{1}{0.5};
        \magogthree{1}{1};
        \iceI{1}{1.5};
        \iceup{1.5}{0.5};
        \magogone{1.5}{1};
        \iceup{1.5}{1.5};
    \end{tikzpicture}
    \ \ \
    \begin{tikzpicture}[scale=1.2]
        \boundcond;
        \iceIV{0.5}{0.5};
        \iceIV{0.5}{1};
        \iceup{0.5}{1.5};
        \iceup{1}{0.5};
        \iceI{1}{1};
        \iceII{1}{1.5};
        \iceIII{1.5}{0.5};
        \iceup{1.5}{1};
        \iceII{1.5}{1.5};
    \end{tikzpicture}
    \ \ \
    \begin{tikzpicture}[scale=1.2]
        \boundcond;
        \iceIV{0.5}{0.5};
        \iceIV{0.5}{1};
        \iceup{0.5}{1.5};
        \iceIV{1}{0.5};
        \iceup{1}{1};
        \iceII{1}{1.5};
        \iceup{1.5}{0.5};
        \iceII{1.5}{1};
        \iceII{1.5}{1.5};
    \end{tikzpicture}
    \caption{The seven magog vertex models of order 3}
    \label{fig:enter-label}
\end{figure}

\begin{lemma}
\label{lem:vertex}
There is an explicit bijection between magog vertex models and magog height-function matrices of order $n$.
\end{lemma}

\begin{proof}
Given a magog fully packed loop, direct the edges so that they go from even to odd vertices. Fill in the rest of the subgraph, and make those edges directed from odd to even vertices. Right-facing structures will become the allowed structure with 3 incoming arrows and left-facing structure will become the allowed structure with 1 incoming arrow. This mapping also proves the fourth condition using the fourth condition of magog fully packed loops. Both forbidden structures in the magog fully packed loop model result in the forbidden structure for magog vertex models. Lastly, applying this bijection to alternating boundary conditions in a magog fully packed loop will result in domain-wall boundary conditions in a magog vertex model. For the reverse direction, we start with a magog vertex model
configuration and ``undo'' each of the steps: keep only those edges which start
at an even vertex and end at an odd vertex then make those edges undirected.
\end{proof}

By combining Lemmas~\ref{lem:cs}, \ref{lem:hf}, \ref{lem:fpl}, and \ref{lem:vertex}, we have proved Theorem~\ref{thm:bijections}.

\end{document}